\documentclass[utf-8,11pt,a4paper]{article}

\usepackage{geometry}
\usepackage{graphicx}
\usepackage{bm}

\usepackage{abstract}
\usepackage{amsmath}
\usepackage{amsfonts}
\usepackage{amsthm}
\usepackage{amssymb}
\usepackage{mathrsfs}
\usepackage[mathscr]{eucal}
\usepackage{eufrak}

\theoremstyle{plain}
\newtheorem{theorem}{Theorem}[section]
\newtheorem{lemma}[theorem]{Lemma}

\newtheorem{proposition}[theorem]{Theorem}
\newtheorem{corollary}[theorem]{Corollary}

\theoremstyle{definition}
\newtheorem{defi}{Definition}[section]
\newtheorem{example}{Examples}[section]

\usepackage[all]{xy}
\usepackage{color}

\newcommand{\eg}{e.g.}

\newcommand{\arsubset}{\ar@{}[r]|-*[@]{\subset}}
\newcommand{\arsupset}{\ar@{}[r]|-*[@]{\supset}}
\newcommand{\arapprox}{\ar@{}[r]|-*[@]{\approx}}
\newcommand{\innerproduct}[2]{\left\langle{#1},{#2}\right\rangle}
\newcommand{\norm}[1]{\|{#1}\|}

\newcommand{\cod}{\mathrm{cod\,}}

\newcommand{\Aut}[1]{\mathrm{Aut}(#1)}

\title{Infinite-Dimensional Generalizations of Orthogonal Groups over Hilbert~Spaces : Constructions~and~Properties}
\author{
Jianwen Luo(\small{E-mail:719418643@qq.com})\\
\small{(High School Affiliated to Shanghai Jiaotong University , Jiading Campus)}\\
Tutor:Yunhe Chen
}
\begin{document}
\begin{titlepage}
\maketitle
\begin{abstract}
In real Hilbert spaces, this paper generalizes the orthogonal groups $\mathrm{O}(n)$ in two ways. One way is by finite multiplications of a family of operators from reflections which results in a group denoted as $\Theta(\kappa)$, the other is by considering the automorphism group of the Hilbert space denoted as $O(\kappa)$. We also try to research the algebraic relationship between the two generalizations and their relationship to the stable~orthogonal~group~$\mathrm{O}=\varinjlim\mathrm{O}(n)$ in terms of topology. In this paper we mainly show that : (a) $\Theta(\kappa)$ is a topological and normal subgroup of $O(\kappa)$; (b) $O^{(n)}(\kappa) \to O^{(n+1)}(\kappa) \stackrel{\pi}{\to} S^{\kappa}$ is a fibre bundle where $O^{(n)}(\kappa)$ is a subgroup of $O(\kappa)$ and $S^{\kappa}$ is a generalized sphere. 
\end{abstract}
\end{titlepage}
\section{Introduction}
Orthogonal groups are very important objects for the linear algebra and Euclidean geometry. In common context we have already defined orthogonal groups of any finite dimension, i.e. $\mathrm{O}(n)$ has been defined for all $n \in \mathbb{N}$. On one hand, $\mathrm{O}(n)$ can be seen as the automorphism group of $\mathbb{R}^{n}$, i.e. $\mathrm{O}(n)=\Aut{\mathbb{R}^{n}}=\{\sigma \in\mathbf{GL}(n) \innerproduct{\sigma\bm{x}}{\sigma\bm{y}}_{\mathbb{R}^{n}}=\innerproduct{\bm{x}}{\bm{y}}_{\mathbb{R}^{n}} \forall \bm{x},\bm{y}\in \mathbb{R}^{n} \}$. On the other hand, $\mathrm{O}(n)$ can be generated by all those reflections which are linear transformations over $\mathbb{R}^{n}$. Both of the statements give us two ways to generalize orthogonal groups to cardinal-dimensional cases: denote the proper class which consists of `all' the cardinal numbers as $\mathbf{Card}$, given a real Hilbert space $\mathcal{H}$ with its dimension $\kappa \in \mathbf{Card}$, one way is to study the automorphism group $\Aut{\mathcal{H}}$, the other is to generate a group with all the reflections which hold the zero vector $\bm{0} \in \mathcal{H}$. For technical reasons we also generalize the conception of `reflections' when practising the second way. Actually we'll get two seem-not-to-be-equivalent versions of `cardinal-dimensional orthogonal groups'. We'll also research some of the groups' algebraic and topological properties.\\\\
This paper will follow these steps:
\begin{enumerate}\addtolength{\itemsep}{-1.5ex}
\item Derive the conception of `reflections' to `reflection operators' by unifying and generalizing a certain family of involutory operators.
\item Given a real Hilbert space $\mathcal{H}$ whose dimension is $\kappa \in \mathbf{Card}$, construct a group namely $\Theta(\kappa)$ with the `reflection operators' as the first generalization of the orthogonal groups.
\item Given a real Hilbert space $\mathcal{H}$ whose dimension is $\kappa \in \mathbf{Card}$, treat its automorphism group $\Aut{\mathcal{H}}$ as the second generalization of the orthogonal groups and denote it as $O(\kappa)$ with a research on the properness of the definition.
\item Discuss the relationship between $\{\Theta(\kappa)\}_{\kappa \in \mathbf{Card}}$ and $\{O(\kappa)\}_{\kappa \in \mathbf{Card}}$.
\item Give out some interesting groups constructed from either $\Theta(\kappa)$ or $O(\kappa)$.
\item Discuss the algebraic properties of $\Theta(\kappa)$ and $O(\kappa)$.
\item Discuss the topological properties of $\Theta(\kappa)$ and $O(\kappa)$, mainly in terms of group actions and fibres.
\end{enumerate}
In the view of sections, the structure of this paper is showed below:\\
\includegraphics[scale=0.55]{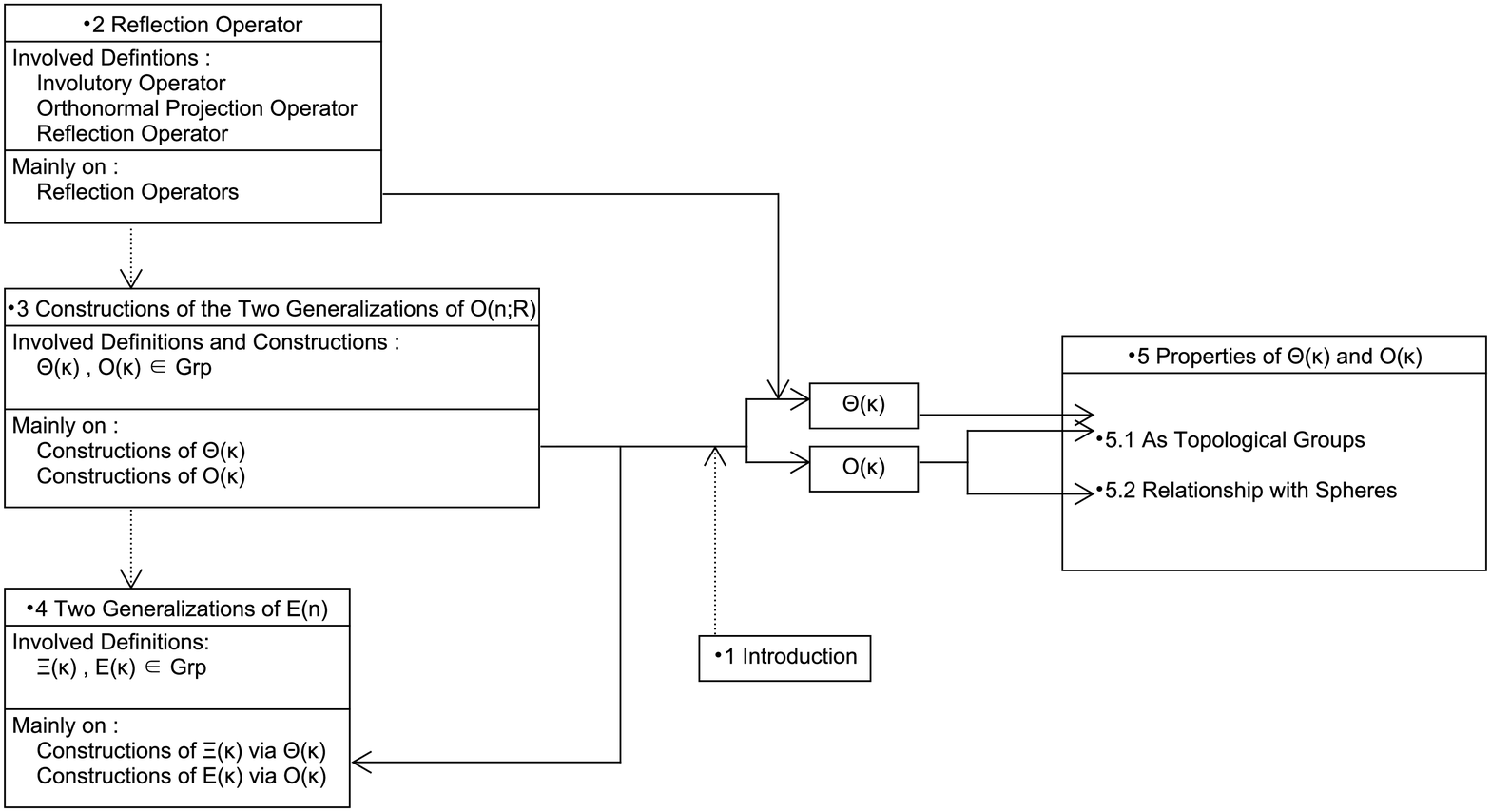}
\newpage
\section{Reflection Operators}
To give out the `bricks' used in the construction of the generalizations, firstly we'll export the definition of `reflection operators'.\\\\
Let's start from several definitions and notations of some special kinds of operators.\\
Followings are the definitions of some standard concepts :
\begin{defi}(\textbf{bounded linear operator and operator norm $\norm{\bullet}_{\text{op}}$})
Given a Hilbert space $\mathcal{H}$ and a linear operator $A$ over $\mathcal{H}$. $A$ is called \textbf{a bounded linear operator} if and only if there exists a constant number $c>0$ such that for every $\bm{v}\in\mathcal{H}$
\[
\norm{A\bm{v}}_{\mathcal{H}} \leq c\norm{\bm{v}}
\]
The smallest such $c$ is called \textbf{the operator norm of $A$}, denoted as $\norm{A}_{\text{op}}$ (or simply $\norm{A}$ if not to be confused). All the bounded linear operators over $\mathcal{H}$ is denoted as $\mathscr{B}(\mathcal{H})$.
\end{defi}
\begin{defi}(\textbf{unitary operator})\\
For a Hilbert space $\mathcal{H}$, a \textbf{unitary operator} $A$ is such an operator whose inverse $A^{-1}$ is exactly its Hilbert conjugate operator $A^{*}$. In other words, the operator $A$ is orthogonal if and only if $A^{*}=A^{-1}$.
\end{defi}
Followings are some definitions of the concepts specified in this paper :
\begin{defi}(\textbf{orthogonal operator})\\
For a real Hilbert space $\mathcal{H}$, an \textbf{orthogonal operator} $A$ is such an operator whose inverse $A^{-1}$ is exactly its Banach conjugate operator $A^{T}$ (and thus its Hilbert conjugate $A^{*}$ since $\mathcal{H}$ is a real Hilbert space). In other words, the operator $A$ is orthogonal if and only if $A^{T}=A^{-1}$.
\end{defi}
\begin{defi}(\textbf{involutory operator})\\
An operator $A$ is \textbf{involutory} if and only if $A^{-1}=A$ , i.e. $A^{2}=I$ .
\end{defi}
\begin{defi}(\textbf{orthogonal projection operator})\\
Let $\mathcal{H}$ be a real Hilbert space and $\mathscr{B}(\mathcal{H})$ the set consists of all of the bounded linear operators over $\mathcal{H}$. An operator $P \in \mathscr{B}(\mathcal{H})$ is an \textbf{orthogonal projection operator} if and only if $P^{2}=P^{T}=P$.
\end{defi}
\begin{proposition}
Given a real Hilbert space $\mathcal{H}$. Given a linear subspace $\Gamma \subseteq \mathcal{H}$ with an orthonormal set as its basis $\{\bm{v}_{\gamma}\}_{\gamma \in \mathfrak{Y}}$. For any vector $\bm{x} \in \mathcal{H}$, the sum $\sum_{\gamma\in\mathfrak{Y}}\innerproduct{\bm{x}}{\bm{v}_{\gamma}} \bm{v}_{\gamma}$ is well defined, i.e. converges to a certain vector in $\mathcal{H}$.
\end{proposition}
\begin{proof}
For the cases $|\mathfrak{Y}|<\infty$, the proposition is obviously true. So we'd just discuss those $\mathfrak{Y}$ which satisfies $|\mathfrak{Y}| \geq \omega$. From Schwarz inequality, $|\innerproduct{\bm{x}}{\bm{v}_{\gamma}}|\leqslant \norm{\bm{x}}\norm{\bm{v}_{\gamma}}=\norm{\bm{x}} \forall \gamma \in \mathfrak{Y} \forall \bm{x} \in \mathcal{H}$.\\
Notice that
\begin{displaymath}
\bigcup_{k=1}^{\infty}[\frac{1}{k+1}\norm{\bm{x}},\frac{1}{k}\norm{\bm{x}}]=(0,\norm{\bm{x}}]
\end{displaymath}
Thus if $|\{\gamma|\innerproduct{\bm{x}}{\bm{v}_{\gamma}} \neq 0\}|>\omega$, then according to pigeon-hole principle, there exists at least one interval where $[\frac{1}{k+1}\norm{\bm{x}},\frac{1}{k}\norm{\bm{x}}]$ includes infinitely-many $\norm{\innerproduct{\bm{x}}{\bm{v}_{\gamma}}}$, in which always exists enough (finitely-many) entries (\eg $(k+1)^{2}+1$) adding up to hold $\sum_{i=1}^{N}\norm{\innerproduct{\bm{x}}{\bm{v}_{\gamma_{i}}}}^{2} > \norm{\bm{x}}^{2}$, contradicting to Bessel inequality.\\As a result, $|\{\gamma|\innerproduct{\bm{x}}{\bm{v}_{\gamma}} \neq 0\}| \leqslant \omega$. Consider the infinite case, sort the countable set as $\{\gamma_{i}\}_{i=1}^{\infty}$.\\
Again refer to Bessel inequality,
\begin{displaymath}
\sum_{i=1}^{N}|\innerproduct{\bm{x}}{\bm{v}_{\gamma_{i}}} |^{2} \leqslant \norm{\bm{x}}^{2} ,\; \forall N \in \mathbb{Z}^{+}
\end{displaymath}
Thus the sum $\sum_{i=1}^{\infty}|\innerproduct{\bm{x}}{\bm{v}_{\gamma_{i}}} |^{2}$ converges.\\
Let $\bm{y}_{n}=\sum_{j=1}^{n}\innerproduct{\bm{x}}{\bm{v}_{\gamma_{j}}} \bm{v}_{\gamma_{j}},n=1,2,\cdots$, when $n>m$,
\begin{displaymath}
\norm{\bm{y}_{n}-\bm{y}_{m}}^{2}=\norm{\sum_{j=m+1}^{n}\innerproduct{\bm{x}}{\bm{v}_{\gamma_{j}}} \bm{v}_{\gamma_{j}}}^{2}=|\sum_{j=m+1}^{n}\innerproduct{\bm{x}}{\bm{v}_{\gamma_{j}}}|^{2}
\end{displaymath}
Thus the array $\{\bm{y}_{n}\}_{n=1}^{\infty}$ is a Cauchy array. For $\mathcal{H}$ is a Hilbert space, its limit exists and lies in $\mathcal{H}$. Denote the limit as $\bm{x}'$.\\
For all $\gamma_{k}$,
\begin{align*}
 \innerproduct{\bm{x}-\bm{x'}}{\bm{v}_{\gamma_{k}}}
&=\innerproduct{\bm{x} - \lim_{n\rightarrow\infty}\sum_{j=1}^{n}\innerproduct{\bm{x}}{\bm{v}_{\gamma_{j}}}\bm{v}_{\gamma_{j}}}{\bm{v}_{\gamma_{k}}}
=\innerproduct{\bm{x}}{\bm{v}_{\gamma_{k}}}-\lim_{n\rightarrow\infty}\innerproduct{\sum_{j=1}^{n}\innerproduct{\bm{x}}{\bm{v}_{\gamma_{j}}}\bm{v}_{\gamma_{j}}}{\bm{v}_{\gamma_{k}}}\\
&=\innerproduct{\bm{x}}{\bm{v}_{\gamma_{k}}}-\lim_{n\rightarrow\infty}[\sum_{j=1}^{n}\innerproduct{\bm{x}}{\bm{v}_{\gamma_{j}}}\innerproduct{\bm{v}_{\gamma_{j}}}{\bm{v}_{\gamma_{k}}}]
=\innerproduct{\bm{x}}{\bm{v}_{\gamma_{k}}}-\innerproduct{\bm{x}}{\bm{v}_{\gamma_{k}}}\\
&=0
\end{align*}
Notice that the given basis is orthonormal while the "component" of "the non-indexed $\gamma$" get $0$ from the inner product with $\bm{x}$, in other words,
\begin{align*}
 \innerproduct{\bm{x}-\bm{x}'}{\bm{v}_{\gamma}}
&=\innerproduct{\bm{x}-\lim_{n\rightarrow\infty}\innerproduct{\bm{x}}{\bm{v}_{\gamma_{j}}}\bm{v}_{\gamma_{j}}}{\bm{v}_{\gamma}}\\
&=\innerproduct{\bm{x}}{\bm{v}_{\gamma}}-\lim_{n\rightarrow\infty}\innerproduct{\sum_{j=1}^{n}\innerproduct{\bm{x}}{\bm{v}_{\gamma_{j}}}\bm{v}_{\gamma_{j}}}{\bm{v}_{\gamma}}
=\innerproduct{\bm{x}}{\bm{v}_{\gamma}}-\sum_{j=1}^{\infty}\innerproduct{\bm{x}}{\bm{v}_{\gamma_{j}}}\innerproduct{\bm{v}_{\gamma_{j}}}{\bm{v}_{\gamma}}\\
&=\innerproduct{\bm{x}}{\bm{v}_{\gamma}}-\innerproduct{\bm{x}}{\bm{v}_{\gamma}}\\
&=0
\end{align*}
Thus $\bm{x}-\bm{x}' \perp \bm{v}_{\gamma} \forall \gamma\in\mathfrak{Y}$.\\
Since the given basis is orthonormal, $\bm{x}-\bm{x}'=\bm{0}$ holds. Therefore,
\begin{align*}
\bm{x}&=\lim_{n\rightarrow\infty}\sum_{j=1}^{n}\innerproduct{\bm{x}}{\bm{v}_{\gamma_{j}}}\bm{v}_{\gamma_{j}}\\
&=\sum_{j=1}^{\infty}\innerproduct{\bm{x}}{\bm{v}_{\gamma_{j}}}\bm{v}_{\gamma_{j}}\\
&=\sum_{\gamma \in \mathfrak{Y}}\innerproduct{\bm{x}}{\bm{v}_{\gamma}}\bm{v}_{\gamma}
\end{align*}
Furthermore, $\norm{\bm{x}}^{2}=\sum_{j=1}^{\infty}|\innerproduct{\bm{x}}{\bm{v}_{\gamma_{j}}}|^{2}=\sum_{\gamma\in\mathfrak{Y}}|\innerproduct{\bm{x}}{\bm{v}_{\gamma}}|^{2}$.\\
\end{proof}
\begin{corollary}
Let $\mathcal{H}$ and $\Gamma$ be defined as above. The operator
\begin{align*}
P_{\Gamma} : H &\to \Gamma\\
\bm{x} &\mapsto \sum_{\gamma\in\mathfrak{Y}}\innerproduct{\bm{x}}{\bm{v}_{\gamma}} \bm{v}_{\gamma}
\end{align*}
is well-defined and it is an orthogonal projection operator.
\end{corollary}
\begin{proof}
From the discussions above, now we just need to verify that $P_{\Gamma}^{*}=P_{\Gamma}$ according to the definition of the Hilbert conjugate operator:
\begin{align*}
\innerproduct{P_{\Gamma}\bm{x}}{\bm{y}}
&=\innerproduct{\sum_{\gamma\in\mathfrak{Y}}\innerproduct{\bm{x}}{\bm{v}_{\gamma}} \bm{v}_{\gamma}}{\bm{y}}
=\sum_{\gamma\in\mathfrak{Y}}\innerproduct{\bm{x}}{\bm{v}_{\gamma}} \innerproduct{\bm{v}_{\gamma}}{\bm{y}}
=\sum_{\gamma\in\mathfrak{Y}}\innerproduct{\bm{y}}{\bm{v}_{\gamma}}\innerproduct{\bm{v}_{\gamma}}{\bm{x}}\\
&=\innerproduct{\sum_{\gamma\in\mathfrak{Y}}\innerproduct{\bm{y}}{\bm{v}_{\gamma}} \bm{v}_{\gamma}}{\bm{x}}
=\innerproduct{P_{\Gamma}\bm{y}}{\bm{x}}
=\innerproduct{\bm{x}}{P_{\Gamma}\bm{y}} \;\forall\bm{x},\bm{y}\in\mathcal{H}
\end{align*}
\end{proof}
\begin{proposition}
Given a real Hilbert space $\mathcal{H}$. Given a linear subspace $\Gamma \subseteq \mathcal{H}$ with orthonormal basis $S_{\mathfrak{A}}=\{\bm{e}_{\alpha}\}_{\alpha \in \mathfrak{A}}$ and $S_{\mathfrak{B}}=\{\bm{\varepsilon}_{\beta}\}_{\beta \in \mathfrak{B}}$. For all $\bm{x} \in H$, the sum $\sum_{\alpha\in\mathfrak{A}}\innerproduct{\bm{x}}{\bm{e}_{\alpha}}\bm{e}_{\alpha}$ and $\sum_{\beta\in\mathfrak{B}}\innerproduct{\bm{x}}{\bm{\varepsilon}_{\beta}}\bm{\varepsilon}_{\beta}$ are equal, i.e. the definition of the operator $P_{\Gamma}$ doesn't depends on the choice of the vector collection.
\begin{proof}
Denote the orthogonal complement of $\Gamma$ as $\Gamma^{\perp}$. Due to projective theorem over Hilbert spaces, for all vector $\bm{x} \in H$,$\bm{x}$ can be decomposed as $\bm{x}=\bm{x}_{1}+\bm{x}_{2}$ where $\bm{x}_{1}\in\Gamma,\bm{x}_{2}\in\Gamma^{\perp}$. Notice that $\bm{x}_{1} \perp \bm{x}_{2}$. Take an orthonormal basis $S_{\mathfrak{Y}}=\{\bm{v}_{\gamma}\}_{\gamma \in \mathfrak{Y}}$ of $\Gamma^{\perp}$, both $S_{\mathfrak{A}} \bigcup S_{\mathfrak{C}}$ and $S_{\mathfrak{B}} \bigcup S_{\mathfrak{C}}$ are orthonormal basis of $H$.
Therefore, 
\[
\sum_{\alpha\in\mathfrak{A}}\innerproduct{\bm{x}}{\bm{e}_{\alpha}}\bm{e}_{\alpha}
=\bm{x}-\sum_{\gamma\in\mathfrak{Y}}\innerproduct{\bm{x}}{\bm{v}_{\gamma}}\bm{v}_{\gamma}
=\sum_{\beta\in\mathfrak{B}}\innerproduct{\bm{x}}{\bm{\varepsilon}_{\beta}}\bm{\varepsilon}_{\beta}
\]
\end{proof}
\end{proposition}
\begin{example}
Let $\mathcal{H}$ be a real Hilbert space.
\begin{enumerate}\addtolength{\itemsep}{-1.5ex}
\item $I:\bm{x} \mapsto \bm{x}$ is an orthogonal projection operator induced by $\mathcal{H}$.
\item $\theta:\bm{x} \mapsto \bm{0}$ is an orthogonal projection operator induced by $\varnothing$.
\end{enumerate}
\end{example}
Now we can give out the definition of the reflection operators :
\begin{defi}(\textbf{reflection operator})\\
Let $\mathcal{H}$ be a real Hilbert space. An operator $L\in\mathscr{B}(\mathcal{H})$ is a \textbf{reflection operator} if and only if there exists such an orthonormal set of $\mathcal{H}$ (namely $\{\bm{e}_{\alpha}\}_{\alpha\in\mathfrak{A}}$) and $P$ the induced orthogonal projection operator $P\in\mathscr{B}(\mathcal{H}):\bm{x}\mapsto\sum_{\alpha\in\mathfrak{A}}\innerproduct{\bm{x}}{\bm{e}_{\alpha}}\bm{e}_{\alpha}$ that satisfying $L=2P-I$.
\end{defi}
\begin{example}
Let $\mathcal{H}$ be a real Hilbert space.
\begin{enumerate}\addtolength{\itemsep}{-1.5ex}
\item For $I:\bm{x}\mapsto\bm{x}$, notice that $I$ is an orthogonal projection operator induced by $\mathcal{H}$ and satisfies $I=2I-I$, thus $I$ is a reflection operator induced by $\mathcal{H}$.
\item For $-I:\bm{x} \mapsto -\bm{x}$, notice that $\theta:\bm{x}\mapsto\bm{0}$ is an orthogonal projection operator induced by $\varnothing$ and satisfies $-I=2\theta-I$, thus $-I$ is a reflection operator induced by $\varnothing$.
\item Consider an orthonormal set $\{\bm{e}_{j}\}_{j=0}^{m}$. If $m=0$, it induces a centrosymmetry; if $m=1$, it induces an axisymmetry ; if $m=2$, it induces a mirror-reflection.
\end{enumerate}
\end{example}
\begin{theorem}
Given a real Hilbert space $\mathcal{H}$. Any reflection operator $L\in\mathscr{B}(\mathcal{H})$ is unitary (and thus orthogonal), self-adjoint and involutory.
\end{theorem}
\begin{proof}
According to the definition of the reflection operators,\\ given any orthonormal set of $\mathcal{H}$ (namely $\{\bm{e}_{\alpha}\}_{\alpha\in\mathfrak{A}}$) and $P$ the induced orthogonal projection operator $P\in\mathscr{B}(\mathcal{H}):\bm{x}\mapsto\sum_{\alpha\in\mathfrak{A}}\innerproduct{\bm{x}}{\bm{e}_{\alpha}}\bm{e}_{\alpha}$, denote the reflection operator as $L=2P-I$.\\
$L$ is obviously linear since both $P$ and $I$ are linear.\\
Since \begin{align*}
\innerproduct{L\bm{x}}{L\bm{x}}&=\innerproduct{2P\bm{x}-\bm{x}}{2P\bm{x}-\bm{x}}=4\innerproduct{P\bm{x}}{P\bm{x}}-4\innerproduct{P\bm{x}}{\bm{x}}+\innerproduct{\bm{x}}{\bm{x}}\\
&=4\innerproduct{P^{*}P\bm{x}}{\bm{x}}-4\innerproduct{P\bm{x}}{\bm{x}}+\innerproduct{\bm{x}}{\bm{x}}=4\innerproduct{P\bm{x}}{\bm{x}}-4\innerproduct{P\bm{x}}{\bm{x}}+\innerproduct{\bm{x}}{\bm{x}}\\
&=\innerproduct{\bm{x}}{\bm{x}} \;\forall\bm{x}\in\mathcal{H}
\end{align*}
thus $L$ holds norm.\\
Notice that $\cod L=L\mathcal{H}=\mathcal{H}$, $L$ is unitary and therefore bounded.\\
In addition, because $L^{*}=2P^{*}-I^{*}=2P^{T}-I=2P-I=L$ , $L$ is self-adjoint.\\
According to the following calculation :
\[
(2P-I)^{2}=4P^{2}-4P+I=4P-4P+I=I
\]
$L$ is involutory.
\end{proof}
As we've proved, the reflection operators performs very well. Now we can use these operators to construct some algebraic structures.For convenience, we denote the set consists of all the reflection operators over a real Hilbert space $\mathcal{H}$ as $\mathrm{Refl}(\mathcal{H})$.
\section{Constructions of the Two Generalizations of $\mathrm{O}(n;\mathbb{R})$}
\subsection{The First Generalization : $\Theta(\kappa)$}
To state our motivation, we give the following lemma :
\begin{lemma}
For any orthogonal matrix $A$, there exists a sequence of reflection operators $\{L_{i}\}_{i=0}^{m}$ that \[A=\prod_{i=0}^{m}L_{i}\]
\end{lemma}
\begin{proof}
Given the orthogonal matrix $A$, any column vector of $A$ is of norm $1$.\\
Transforms the first column of $A$. Since the column vector is orthogonal thus shares the same norm with $\begin{bmatrix}1,0,\cdots,0\end{bmatrix}^{T}$, exists $L_{1}$ (not unique) which leads to
 \begin{displaymath}
 L_{1}A=\begin{bmatrix}
 1 & \star \\
 0 & A_{1}
 \end{bmatrix}
 \end{displaymath}
 Since $L_{1}$ and $A$ are both unitary, $L_{1}A$ is unitary. Thus
 \begin{displaymath}
 \begin{bmatrix}
 1 & 0 \\
 0 & I_{n-1}
 \end{bmatrix}=\begin{bmatrix}
 1 & \star \\
 0 & A_{1}
 \end{bmatrix}
 \begin{bmatrix}
 1 & 0 \\
 \star & A_{1}^{T}
 \end{bmatrix}
 =\begin{bmatrix}
 1+\star^{2} & \star A_{1}^{T} \\
 A_{1} \star & A_{1}A_{1}^{T}
 \end{bmatrix}
 \end{displaymath}
 It implies that $\star =0$ and $A_{1}A_{1}^{T}=I$. Furthermore, because
 \begin{displaymath}
 \begin{bmatrix}
 1 & 0 \\
 0 & I_{n-1}
 \end{bmatrix}=
 \begin{bmatrix}
 1 & 0 \\
 \star & A_{1}^{T}
 \end{bmatrix}
 \begin{bmatrix}
 1 & \star \\
 0 & A_{1}
 \end{bmatrix}
 =\begin{bmatrix}
 1 & 0 \\
 0 & A_{1}^{T}A_{1}
 \end{bmatrix}
 \end{displaymath}
 we have $A_{1}^{T}A_{1}=A_{1}A_{1}^{T}=I$, i.e. $A_{1}$ is an orthogonal matrix whose type is $(n-1)\times(n-1)$.\\
 So
 \begin{displaymath}
 L_{1}A=\begin{bmatrix}
 1 & 0 \\
 0 & A_{1}
 \end{bmatrix}
 \end{displaymath}
 Again we transforms $A_{1}$ similarly, we can get $L_{2}'$ which holds
 \begin{displaymath}
 L_{2}'A_{1}=\begin{bmatrix}
 1 & 0 \\
 0 & A_{2}
 \end{bmatrix}
 \end{displaymath}
 Set $L_{2}$ to be
 \begin{displaymath}
 L_{2}=\begin{bmatrix}
 1 & 0 \\
 0 & L_{2}'
 \end{bmatrix}
 \end{displaymath}
 As we've proved, $L_{2}$ is a linear reflection operator. 
 Now we have
 \begin{displaymath}
 L_{2}L_{1}A=\begin{bmatrix}
 I_{2} & 0 \\
 0 & A_{2}
 \end{bmatrix}
 \end{displaymath}
 Apply the algorithm recursively, we can get $\{L_{i}\}_{i=1}^{k}(k \leqslant n)$ which holds the equality $(\prod_{i=1}^{k}L_{i})A=I$.\\
Multiply $A^{T}=A^{-1}$ to both sides, it gives out $\prod_{i=1}^{k}L_{i}=A^{T}$. Since all the involved matrices are self-adjoint, $A=\prod_{i=1}^{k}L_{k-i+1}$.
\end{proof}
This lemma show us the structure of the orthogonal matrices and thus the structure of the orthogonal group $\mathrm{O(n)}$  in the finite-dimensional cases. Or to be more precise,
\begin{corollary}
The set $\mathrm{Refl}(\mathbb{R}^{n})$ generates $\mathrm{O}(n;\mathbb{R})$.
\end{corollary}
From this conclusion, we construct the following set :
\begin{defi}(\textbf{The First Generalization : $\Theta(\kappa)$})\\
Let $\mathcal{H}$ be a real Hilbert space with $\dim\mathcal{H}=\kappa\in\mathbf{Card}$, then $\Theta(\kappa)$ is defined as the set consists of the combinations of finitely many elements in $\mathrm{Refl}(\mathcal{H})$ under the usual operator multiplication :
\[
\Theta(\kappa)=\{A\in\mathscr{B}(\mathcal{H}) : \exists\{L_{j}\}_{j=1}^{m}\subset\mathrm{Refl}(\mathcal{H})((m<\omega)\wedge(A=\prod_{j=1}^{m}\L_{j}))\}
\]
\end{defi}
According to the fact that real Hilbert spaces with the same Hilbert-dimension are isomorphic as inner product spaces, $\Theta(\kappa)$ is well defined.
What's more, the set constructed above is actually a group~:
\begin{proposition}
For all $\kappa \in \mathbf{Card}$,$\Theta(\kappa)$ forms a group.
\end{proposition}
\begin{proof}
We'll verify the proposition according to the definition of `group' step by step.
\item[Closeness :]
The product of two products of flat-reflection operators is a product of a longer flat-reflection operators sequence. Thus $\Theta(\kappa)$ closes under operator multiplication.
\item[Associativity :]Inherited from the operator multiplication.
\item[Identity :]Recall the fact that the identity operator $I$ is also a reflection operator which is induced by $\mathcal{H}$. Therefore we have $I \in \Theta(\kappa)$ as its identity.
\item[Invertibility :]As we've proved, every reflection operator is involutory, thus for any operator $A \in \Theta(\kappa)$ with the decomposition $A=\prod_{j=1}^{k}L_{j}$, we have its inverse $A^{-1}=\prod_{j=1}^{k}L_{k-j+1}\in\Theta(\kappa)$.
\end{proof}
\begin{corollary}
If $\kappa = n < \infty$, $\Theta(\kappa) = \Theta(n) \cong \mathrm{O}(n)$.
\end{corollary}
From now on, we consider the group $\Theta(\kappa)$ as the first generalization of $\mathrm{O}(n)$.
\subsection{The Second Generalization : $\mathrm{O}(\kappa)$}
First, consider the definition of the automorphism groups of Hilbert spaces :
\begin{defi}(\textbf{automorphism group of a real Hilbert space})\\
\indent Given a real Hilbert space $\mathcal{H}$, its automorphism group $\Aut{\langle\,\mathcal{H},+,\cdot,\innerproduct{}{}\,\rangle}=\Aut{\mathcal{H}}$ is defined to be the largest subgroup of $\Aut{\langle\,\mathcal{H},+\,\rangle}$ which satisfies all of the following requirements:
\begin{enumerate}
\item $\forall\sigma\in\Aut{\mathcal{H}}$, $\sigma:\mathcal{H}\to\mathcal{H}$ is bijective.
\item $\sigma(\bm{x})+\sigma(\bm{y})=\sigma(\bm{x}+\bm{y}) \;\forall\sigma\in\Aut{\mathcal{H}}\forall\bm{x},\bm{y}\in\mathcal{H}$
\item $\sigma(k\bm{x})=k\sigma(\bm{x}) \;\forall\sigma\in\Aut{\mathcal{H}}\forall{k}\in\mathbb{R}\forall\bm{x}\in\mathcal{H}$
\item $\innerproduct{\sigma(\bm{x})}{\sigma(\bm{y})}=\innerproduct{\bm{x}}{\bm{y}} \;\forall\sigma\in\Aut{\mathcal{H}}\forall\bm{x},\bm{y}\in\mathcal{H}$
\end{enumerate}
Terminologically,such an element $\sigma\in\Aut{\mathcal{H}}$ is called an automorphism of $\mathcal{H}$.
\end{defi}
Notice that $\mathbb{R}^{n}$ is a special case of real Hilbert spaces in finite dimensions, we have :
\begin{theorem}
$\Aut{\mathbb{R}^{n}}=\mathrm{O}(n)\;\forall{n}\in\mathbb{N}$
\end{theorem}
\begin{proof}
Obviously $\mathrm{O}(n)\subseteq\Aut{\mathbb{R}^{n}}$.\\
On the opposite direction, $\forall\sigma\in\Aut{\mathbb{R}^{n}}\forall\bm{x},\bm{y}\in\mathbb{R}^{n}$,
\begin{align*}
&\innerproduct{\sigma\bm{x}}{\bm{y}}=\innerproduct{\sigma^{-1}\sigma\bm{x}}{\sigma^{-1}\bm{y}}=\innerproduct{\bm{x}}{\sigma^{-1}\bm{y}}\\
&\innerproduct{\sigma\bm{x}}{\bm{y}}=\innerproduct{\bm{x}}{\sigma^{*}\bm{y}}=\innerproduct{\bm{x}}{\sigma^{T}\bm{y}}
\end{align*}
which indicates that $\sigma\in\mathrm{O}(n)\;\forall\sigma\in\Aut{\mathbb{R}^{n}}$ according to the uniqueness of an operator's Hilbert conjugate. Thus $\Aut{\mathbb{R}^{n}}\subseteq\mathrm{O}(n)$.\\
Therefore the underlying sets $\Aut{\mathbb{R}^{n}}=\mathrm{O}(n)$. Recall the fact that the two groups share the same group operation, they're identical as groups.
\end{proof}
Generalizing to infinite-dimensional real Hilbert spaces, we have :
\begin{theorem}
Given a real Hilbert space $\mathcal{H}$ with any dimension,
\[
\{ A \in \mathscr{B}(\mathcal{H}) : A^{-1}=A^{T}\}=\Aut{\mathcal{H}}
\]
\end{theorem}
\begin{proof}
Denote the mentioned set as $S$. Obviously $S\subseteq\Aut{\mathcal{H}}$.\\
On the opposite direction, $\forall\sigma\in\Aut{\mathcal{H}}\forall\bm{x},\bm{y}\in\mathcal{H}$,
\begin{align*}
&\innerproduct{\sigma\bm{x}}{\bm{y}}=\innerproduct{\sigma^{-1}\sigma\bm{x}}{\sigma^{-1}\bm{y}}=\innerproduct{\bm{x}}{\sigma^{-1}\bm{y}}\\
&\innerproduct{\sigma\bm{x}}{\bm{y}}=\innerproduct{\bm{x}}{\sigma^{*}\bm{y}}=\innerproduct{\bm{x}}{\sigma^{T}\bm{y}}
\end{align*}
which indicates that $\sigma\in{S}\;\forall\sigma\in\Aut{\mathcal{H}}$ according to the uniqueness of an operator's Hilbert conjugate. Thus $\Aut{\mathcal{H}}\subseteq{S}$.\\
Therefore the underlying set $\Aut{\mathcal{H}}=S$, i.e. $\{ A \in \mathscr{B}(\mathcal{H}) : A^{-1}=A^{T}\}=\Aut{\mathcal{H}}$.
\end{proof}
To make the second generalization to be given well-defined, we give out the following lemma :
\begin{lemma}
Let $\mathcal{H}_{1}$ and $\mathcal{H}_{2}$ be two real Hilbert spaces.
\[
\dim\mathcal{H}_{1}=\dim\mathcal{H}_{2} \iff \Aut{\mathcal{H}_{1}}\cong\Aut{\mathcal{H}_{2}}
\]
\end{lemma}
\begin{proof}
Immediately followed after the fact that $\dim\mathcal{H}_{1}=\dim\mathcal{H}_{2} \iff \mathcal{H}_{1}\cong\mathcal{H}_{2}$.
\end{proof}
Here comes the second generalization of the orthogonal groups :
\begin{defi}(\textbf{The Second Generalization : $O(\kappa)$}):\\
Let $\kappa\in\mathbf{Card}$ and $\mathcal{H}$ a real Hilbert space satisfying $\dim\mathcal{H}=\kappa$. $O(\kappa)$ is then defined as the automorphism group of $\mathcal{H}$, i.e. $O(\kappa)=\Aut{\mathcal{H}}$.
\end{defi}
According to the lemma given above, $O(\kappa)$ is unique up to isomorphisms.And as we've proved,
\begin{theorem}
\[
O(n) \cong \mathrm{O}(n) \;\forall n\in\mathbb{N}
\]
where $\mathrm{O}(n)$ is the common orthogonal group and $O(n)$ the generalization in finite-dimensional cases.
\end{theorem}
Notice that $O(\omega)$ (and even $\Theta(\omega)$) are different from the stable orthogonal group $\mathrm{O}=\varinjlim{\mathrm{O}(n)}=\bigcup_{n=0}^{\infty}\mathrm{O}(n)$, e.g. consider the reflection induced by $\bm{x}=\{x_{i}=\frac{1}{2^{i}}\}_{i=1}^{\infty}$ which lies in $\Theta(\omega)$ (and thus in $O(\omega)$) but not in $\mathrm{O}$. To be precise, $\mathrm{O}\subset\Theta(\omega)\trianglelefteq{O(\omega)}$ as we'll see.
\subsection{The At-First-Sight Relationship between $\Theta(\kappa)$ and $O(\kappa)$}
\begin{theorem}
For any $\kappa \in \mathbf{Card}$, $\Theta(\kappa)$ is an normal subgroup of $O(\kappa)$, i.e.
$
\Theta(\kappa) \trianglelefteq O(\kappa) \forall \kappa \in \mathbf{Card}
$.
\end{theorem}
\begin{proof}
First, we should prove that $\Theta(\kappa) \subseteq O(\kappa)$. Since every reflection operator is orthogonal, so is their product which satisfies the requirement above. Thus $\Theta(\kappa)$ is a subgroup of $O(\kappa)$.\\
Now consider an operator $A \in \Theta(\kappa)$ which can be written as $A=\prod_{i=0}^{m}L_{i}$ where $L_i(i=0,\cdots,m)$ are reflection operators. Let $X \in O(\kappa)$. Then\[
X^{-1}AX=X^{-1}(\prod_{i=0}^{m}L_{i})X=\prod_{i=0}^{m}X^{-1}L_{i}X=\prod_{i=0}^{m}X^{T}L_{i}X
\]
Notice that by definition $L_{i}(i=0,\cdots,m)$ has the form $\bm{x}\mapsto -I\bm{x}+2\sum_{\alpha \in \mathfrak{A}}\innerproduct{\bm{x}}{\bm{e}_{\alpha}}\bm{e}_{\alpha}$ where $\{\bm{e}_{\alpha}\}_{\alpha\in\mathfrak{A}}$ is an orthonormal vector collection. Thus \[X^{T}L_{i}X:\bm{x}\mapsto -X^{T}IX\bm{x}+2\sum_{\alpha \in \mathfrak{A}}X^{T}\innerproduct{X\bm{x}}{\bm{e}_{\alpha}}\bm{e}_{\alpha}=-\bm{x}+2\sum_{\alpha \in \mathfrak{A}}\innerproduct{\bm{x}}{X^{T}\bm{e}_{\alpha}}(X^{T}\bm{e}_{\alpha})\]
It implies that $X^{-1}L_{i}X=X^{T}L_{i}X (i=0,\cdots,m)$ are also reflection operators, i.e. \[
X^{-1}AX \in \Theta(\kappa) \subseteq O(\kappa) \forall A\in\Theta(\kappa)\forall X \in O(\kappa)
\]
In one word, $
\Theta(\kappa) \trianglelefteq O(\kappa) \forall \kappa \in \mathbf{Card}
$.
\end{proof}
Due to this theorem, $\Theta(\kappa)$ seems to have a stronger requirement than $O(\kappa)$.
\section{Two Generalizations of the Euclidean Groups}
\subsection{Motivation}
As a demo, we'd like to generalize the Euclidean groups $\mathrm{E}(n)$ by the way. For convenience, first we give out the definition of the translation groups $\mathrm{T}(n)$ for finite-dimensional cases.
\begin{defi}(\textbf{Translation Group of finite dimension :} $\mathrm{T}(n)$)\\
\indent For finite dimension $n$, the \textbf{translation group of dimension $n$ namely $\mathrm{T}(n)$} is defined as the additive group structure of $n$-dimensional Euclidean space $\mathbb{R}^{n}$, i.e. 
\[\mathrm{T}(n)=\langle \mathbb{R}^{n} , + \rangle\]
\end{defi}
In this section, the homomorphism $\varphi$ is defined as follows (which will be frequently used soon):
\begin{align*}
\varphi : \mathrm{O}(n) &\rightarrow \Aut{\mathrm{T}(n)}\\
M &\mapsto (\varphi_{M}:\bm{x}\mapsto M\bm{x})
\end{align*}
This homomorphism induces an operation $\cdot_{\varphi}$ on $\mathrm{T}(n) \times \mathrm{O}(n)$ which makes the Cartesian product be a group.
\begin{align*}
\cdot_{\varphi} : (\mathrm{T}(n)\times\mathrm{O}(n))\times(\mathrm{T}(n)\times\mathrm{O}(n))&\rightarrow(\mathrm{T}(n)\times\mathrm{O}(n))\\
((\bm{x}_{1},M_{1}),(\bm{x}_{2},M_{2})) &\mapsto (\bm{x}_{1}+\varphi_{M_{1}}(\bm{x}_{2}),M_{1}M_{2})=(\bm{x}_{1}+M_{1}\bm{x}_{2},M_{1}M_{2})
\end{align*}
This leads us to the semidirect product $\mathrm{T}(n) \rtimes_{\varphi} \mathrm{O}(n)=\langle \mathrm{T}(n)\times\mathrm{O}(n) \, , \, \cdot_{\varphi} \rangle$ where `$\times$' denotes the Cartesian product. In fact, it's exactly the Euclidean group $\mathrm{E}(n)=\mathrm{ISO}(\mathbb{R}^{n})$ with action $\mu$ on $\mathbb{R}^{n}$:
\begin{align*}
\mu : \mathrm{E}(n) \times \mathbb{R}^{n} &\rightarrow \mathbb{R}^{n}\\
((M,\bm{v}),\bm{x})&\mapsto M\bm{x}+\bm{v}
\end{align*}
The construction above can be easily generalized via either $\Theta(\kappa)$ or $O(\kappa)$ since $\mathrm{O}(n)$ is just their special case for finite dimension. This statement leads to the following two generalizations of Euclidean groups on infinite-dimensional Hilbert spaces.
\subsection{Constructions}
Since $\mathbb{R}^{n}$ can be seen as a real finite-dimensional Hilbert space, the translation group can be defined on infinite dimensions as following:
\begin{defi}(\textbf{Translation Group of any dimension :} $\mathrm{T}(\kappa)$)\\
\indent For any dimension $\kappa \in \mathbf{Card}$, the \textbf{translation group of dimension $\kappa$ namely $\mathrm{T}(\kappa)$} is defined as the additive group structure of $\kappa$-dimensional real Hilbert space $\mathcal{H}$, i.e. 
\[\mathrm{T}(n)=\langle \mathcal{H} , + \rangle ,\;(\dim \mathcal{H} = \kappa)\]
which is unique up to isomorphism.
\end{defi}
Notice that $\mathrm{O}(n)$ is the special case of $\Theta(\kappa)$ and $O(\kappa)$ on finite dimensions. Therefore, $\varphi$ can be extended to $O(\kappa)$ and then restricted to $\Theta(\kappa)$. For convenience we still use the notation `$\varphi$' :
\begin{align*}
\varphi : O(\kappa) &\rightarrow \Aut{\mathrm{T}(\kappa)}\\
M &\mapsto (\varphi_{M}:\bm{x}\mapsto M\bm{x})
\end{align*}
This homomorphism also induces an operation $\cdot_{\varphi}$ on $\mathrm{T}(\kappa) \times \mathrm{O}(\kappa)$ which makes the set another group.
\begin{align*}
\cdot_{\varphi} : (\mathrm{T}(\kappa)\times O(\kappa))\times(\mathrm{T}(\kappa)\times O(\kappa))&\rightarrow(\mathrm{T}(\kappa)\times O(\kappa))\\
((\bm{x}_{1},M_{1}),(\bm{x}_{2},M_{2})) &\mapsto (\bm{x}_{1}+\varphi_{M_{1}}(\bm{x}_{2}),M_{1}M_{2})=(\bm{x}_{1}+M_{1}\bm{x}_{2},M_{1}M_{2})
\end{align*}
The semidirect product $\mathrm{T}(\kappa)\rtimes_{\varphi} O(\kappa)$ induces a generalization of the Euclidean groups. Restrict $\varphi$ on $\Theta(\kappa)$, the semidirect product $\mathrm{T}(\kappa)\rtimes_{\varphi} \Theta(\kappa)$ offers another generalization.
\begin{defi}(\textbf{Two Generalizations of Euclidean Group :} $\Xi(\kappa)$ \textbf{and} $E(\kappa)$)\\
\indent The homomorphism $\varphi$ is the same as given above and let $\kappa \in \mathbf{Card}$. The \textbf{broken Euclidean group} $\Xi(\kappa)$ is defined as the semidirect product $\Xi(\kappa)=\mathrm{T}(\kappa)\rtimes_{\varphi} \Theta(\kappa)$. The \textbf{full Euclidean group} $E(\kappa)$ is defined as the semidirect product $E(\kappa)=\mathrm{T}(\kappa)\rtimes_{\varphi} O(\kappa)$.
\end{defi}
Their odd names are from the following theorem:
\begin{theorem}
For any $\kappa\in\mathbf{Card}$,$\Xi(\kappa)$ is a normal subgroup of $E(\kappa)$, i.e. : $\Xi(\kappa) \trianglelefteq E(\kappa)\,\forall\kappa\in\mathbf{Card}$.
\end{theorem}
\begin{proof}
Let $\mathcal{H}$ be the $\kappa$-dimensional real Hilbert space while $(\bm{a},A) \in \Xi(\kappa)$ and $(\bm{m},M)\in E(\kappa)$.\\
Notice that $(\bm{m},M)^{-1}=(-M^{-1}\bm{m},M^{-1})$. Therefore, \[
(\bm{m},M)^{-1}(\bm{a},A)(\bm{m},M)=(-M^{-1}\bm{m},M^{-1})(\bm{a},A)(\bm{m},M)=(-M^{-1}\bm{m}+M^{-1}\bm{a}+M^{-1}A\bm{m},M^{-1}AM)
\]
Due to this equation where $-M^{-1}\bm{m}+M^{-1}\bm{a}+M^{-1}A\bm{m} \in \mathcal{H}$ thus $-M^{-1}\bm{m}+M^{-1}\bm{a}+M^{-1}A\bm{m} \in \mathrm{T}(\kappa)$,whether $(\bm{m},M)^{-1}(\bm{a},A)(\bm{m},M)$ is in $\Xi(\kappa)$ only depends on whether $M^{-1}AM$ is in $\Theta(\kappa)$. Since $\Theta(\kappa) \trianglelefteq O(\kappa)$, $M^{-1}AM$ is in $\Theta(\kappa)$ thus $(\bm{m},M)^{-1}(\bm{a},A)(\bm{m},M)$ is in $\Xi(\kappa)$. This implies that $\Xi(\kappa)$ is a normal subgroup of $E(\kappa)$, i.e. $\Xi(\kappa) \trianglelefteq E(\kappa)$.
\end{proof}
What's more, as a corollary of Mazur-Ulam theorem, we have :
\begin{corollary}
Let $\mathcal{H}$ be a real Hilbert space whose dimension is $\kappa \in \mathbf{Card}$, then $E(\kappa)$ is isomorphic to its isometry group $\mathrm{ISO}(\mathcal{H})$, i.e. $E(\kappa) \cong \mathrm{ISO}(\mathcal{H})$.
\end{corollary}
Compared to the finite-dimensional cases where $\mathrm{E}(n)=\mathrm{ISO}(\mathbb{R}^{n})$ , $E(\kappa)$ seems to be a good enough generalization of the Euclidean groups.

\section{Properties of $\Theta(\kappa)$ and $O(\kappa)$}
\subsection{As Topological Groups}
\begin{theorem}
With the topology induced by $\norm{\bullet}_{\mathrm{op}}$, $\Theta(\kappa)$ and $O(\kappa)$ are topological groups.
\end{theorem}
\begin{proof}
For $A,B\in \mathscr{B}(\mathcal{H})$, \[
\norm{AB}_{\mathrm{op}} = \mathop{\sup}_{\norm{\bm{x}}_{\mathcal{H}}=1}\norm{AB\bm{x}}_{\mathcal{H}}\leq\norm{A}_{\mathrm{op}}\norm{B}_{\mathrm{op}}\norm{\bm{x}}_{\mathcal{H}}=\norm{A}_{\mathrm{op}}\norm{B}_{\mathrm{op}}
\]
That is $\norm{AB}_{\mathrm{op}}\leq\norm{A}_{\mathrm{op}}\norm{B}_{\mathrm{op}}$ which implies the continuity of the operation on $\mathscr{B}(\mathcal{H})$. Since $\Theta(\kappa)\trianglelefteq{O(\kappa)}\subset \mathscr{B}(\mathcal{H})$, the operation restricted on the two groups is also continuous, i.e. $\Theta(\kappa)$ and $O(\kappa)$ are topological groups with the topology induced by $\norm{\bullet}_{\mathrm{op}}$.
\end{proof}
\begin{corollary}
$\Theta(\kappa)$ is a topological subgroup of $O(\kappa)$ with the topology induced by $\norm{\bullet}_{\mathrm{op}}$.
\end{corollary}
\subsection{Topological Relationship with Spheres}
About the topology of $\Theta(n)=O(n)=\mathrm{O}(n)$ we have the lemma:
\begin{lemma}
The following graph consists of fibre bundles
\[
\xymatrix{
&S^{0}&S^{1}&S^{2}&S^{3}&\cdots&S^{n}&\cdots\\
0\ar[r]&\mathrm{O}(1)\ar[r]\ar[u]^{\pi_{0}}&\mathrm{O}(2)\ar[r]\ar[u]^{\pi_{1}}&\mathrm{O}(3)\ar[r]\ar[u]^{\pi_{2}}&\mathrm{O}(4)\ar[r]\ar[u]^{\pi_{3}}&\cdots\ar[r]&\mathrm{O}(n+1)\ar[r]\ar[u]^{\pi_{n}}&\cdots
}
\]
where each fibre bundle $(E,B,\pi,F)$ is denoted as
\[
\xymatrix{
&B\\
F\ar[r]&E\ar[u]^{\pi}
}
\]
with $j=0,1,\cdots,n,n+1,\cdots$, and $S^{j}$ the spheres, $F$ the fibre, $E$ the total space, $B$ the base space.
\end{lemma}
\begin{proof}
Let $V_{k}(\mathbb{R}^{n})$ denote the Stiefel manifold.\\
First we take an orthonormal basis of $\mathbb{R}^{j+1}$ namely $\{\bm{e}_{i}\}_{i=1}^{j+1}\in{V_{j+1}(\mathbb{R}^{j+1})}$. Identify $\mathbb{R}^{j}$ with $\mathsf{span}(\{\bm{e}_{i}\}_{i=2}^{j+1})$ as following : take an orthonormal basis of $\mathbb{R}^{j}$ anonymously with a bijection $\eta$ from the anonymous basis ($\in{V_{j}(\mathbb{R}^{j})}$) to $\{\bm{e}_{i}\}_{i=2}^{j+1}\in{V_{j}(\mathbb{R}^{j+1})}$ and linearly span the bijection, denoting the resulting isomorphism as $\eta : \mathsf{span}(\{\bm{e}_{i}\}_{i=2}^{j+1}) \to \mathbb{R}^{j}$.\\
Let the continuous surjection $\pi_{j} : \mathrm{O}(j+1) \to S^{j}$ be defined as follows :
\begin{align*}
\pi_{j} : \mathrm{O}(j+1) &\to S^{j}\\
A &\mapsto A\bm{e}_{1}
\end{align*}
We just need to verify for every $\bm{x} \in S^{j}$, there exists a neighbourhood $\mathcal{N}_{\bm{x}}$ of $\bm{x}$ and a homeomorphism $\varphi : \pi^{-1}(\mathcal{N}_{\bm{x}}) \to \mathcal{N}_{\bm{x}}\times\mathrm{O}(j)$ making the following graph commute :
\[
\xymatrix{
\pi^{-1}_{j}(\mathcal{N}_{\bm{x}})\ar[d]^{\pi_{j}}\ar@{..>}[r]^{\varphi}&\mathcal{N}_{\bm{x}}\times\mathrm{O}(j)\ar[dl]^{\text{proj}_{1}}\\
\mathcal{N}_{\bm{x}}
}
\]
where $\text{proj}_{1} : \mathcal{N}_{\bm{x}}\times\mathrm{O}(j) \to \mathcal{N}_{\bm{x}}$ is the canonical projection to the first factor.\\
For every $A,B\in\mathrm{O}(j+1)$ which satisfy $A\bm{e}_{1}=B\bm{e}_{1}$, we have :
\[
\mathsf{span}(\{A\bm{e}_{i}\}_{i=2}^{j+1})=(\mathsf{span}(\{A\bm{e}_{1}\}))^{\perp}=(\mathsf{span}(\{B\bm{e}_{1}\}))^{\perp}=\mathsf{span}(\{B\bm{e}_{i}\}_{i=2}^{j+1})
\]
In other words, for every $\bm{x}\in{S^{j}}$, a $j$-dimensional subspace is associated. At every point $\bm{x}\in{S^{j}}$, take an orthonormal basis of the subspace $(\mathsf{span}(\{\bm{x}\}))^{\perp}=(\mathsf{span}(\{A\bm{e}_{1}\}))^{\perp}$ namely $\{\bm{\varepsilon}_{i}\}_{i=1}^{j}\in{V_{j}(\mathbb{R}^{j+1})}$. Recall our identification for $\mathsf{span}(\{\bm{e}_{i}\}_{i=2}^{j+1})$ and $\mathbb{R}^{j}$. Consider the bijection :
\begin{align*}
f : \{\bm{\varepsilon}_{i}\}_{i=1}^{j} &\to \left\{\eta(\bm{e}_{i})\right\}_{i=2}^{j+1}\in{V_{j}(\mathbb{R}^{j})} \\
\bm{\varepsilon}_{i} &\mapsto \eta(\bm{e}_{i+1})
\end{align*}
Linearly span the bijection $f$ to the isomorphism (and thus a homeomorphism) $\hat{f} : \mathsf{span}(\{\bm{\varepsilon}_{i}\}_{i=1}^{j}) \to \mathbb{R}^{j}$.
Therefore for every $A\in\pi^{-1}(\bm{x})$, $\left\{\hat{f}(A\bm{\varepsilon}_{i})\right\}_{i=2}^{j+1}\in{V_{j}(\mathbb{R}^{j})}$ is an orthonormal basis of $\mathbb{R}^{j}$ while $\left\{\hat{f}(\bm{\varepsilon}_{i})\right\}_{i=2}^{j+1}\in{V_{j}(\mathbb{R}^{j})}$ is also an orthonormal basis of $\mathbb{R}^{j}$.\\
According to the fact that $A$ can be written as the following form uniquely with $\{\bm{e}_{i}\}_{i=1}^{j+1}$ given :
\begin{align*}
A : \mathbb{R}^{j+1} &\to \mathbb{R}^{j+1}\\
 \bm{x} &\mapsto \sum_{i=1}^{j+1}\innerproduct{\bm{x}}{\bm{e}_{i}}(A\bm{e}_{i})
\end{align*}
For every point $\bm{x}\in{S^{j}}$, we can define this mapping :
\begin{align*}
\phi_{\bm{x}} : \pi^{-1}(\bm{x}) &\to \mathrm{O}(j) \\
(A:\bm{x} \mapsto \sum_{i=1}^{j+1}\innerproduct{\bm{x}}{\bm{e}_{i}}(A\bm{e}_{i})) &\mapsto (\phi_{\bm{x}}(A):\bm{x} \mapsto \sum_{i=1}^{j}\innerproduct{\bm{x}}{\hat{f}(\bm{\varepsilon}_{i})}(\hat{f}(A\bm{\varepsilon}_{i})))
\end{align*}
where $\{\bm{\varepsilon}_{i}\}_{i=1}^{j}$ is the orthonormal basis we have taken above for every $\bm{x}\in{S^{j}}$.\\
Now we can give out the homeomorphism $\varphi : \pi^{-1}(\mathcal{N}_{\bm{x}}) \to \mathcal{N}_{\bm{x}}\times\mathrm{O}(j)$ :
\begin{align*}
\varphi : \pi^{-1}(\mathcal{N}_{\bm{x}}) &\to \mathcal{N}_{\bm{x}}\times\mathrm{O}(j) \\
A &\mapsto (A\bm{e}_{1},\phi_{A\bm{e}_{1}}(A))
\end{align*}
Let's prove that $\varphi$ is a homeomorphism : \\
\indent It's obvious that $\varphi$ is a surjection. Consider $\pi^{-1}({\bm{x}})/(\text{proj}_{2}\circ\varphi)$, in each equivalent class $[A]$ for any element $A\in[A]$, $\{A\bm{e}_{i}\}_{i=2}^{j+1}$ identifies while $A\bm{e}_{1}$ differs, i.e. $\varphi$ is injective. That is, $\varphi$ is a bijection.\\
\indent The `evaluation' $\mathrm{O}(j+1)\to\mathbb{R}^{j+1};A\mapsto{A\bm{e}_{1}}$ is continuous, i.e. $\varphi$ is continuous for the first component.\\
\indent $\forall A_{1},A_{2} \in \pi^{-1}(\bm{x})$ , notice that $A_{1}\bm{e}_{1}=A_{2}\bm{e}_{1}=\bm{x}$. Consider $\tilde{A_{1}}=\phi_{\bm{x}}(A_{1})$ and $\tilde{A_{2}}=\phi_{\bm{x}}(A_{2})$ , which have the property :
\begin{align*}
\tilde{A}\hat{f}(\bm{\varepsilon}_{i})=\phi_{\bm{x}}(A)\hat{f}(\bm{\varepsilon}_{i})=\innerproduct{\hat{f}(\bm{\varepsilon}_{i})}{\hat{f}(\bm{\varepsilon}_{i})}(\hat{f}(A\bm{\varepsilon}_{i})))=\hat{f}(A\bm{\varepsilon}_{i})\;\;\;(i=1,2,\cdots,j\,,\,\forall{A\in\pi^{-1}(\bm{x})})
\end{align*}
\indent To compute the distance between $A_{1}$ and $A_{2}$, we have :
\begin{align*}
d(A_{1},A_{2})&=\norm{A_{1}-A_{2}}_{\text{op}}=\inf_{\bm{v}\in S^{j}}\norm{A_{1}\bm{v}-A_{2}\bm{v}}_{\mathbb{R}^{j+1}}\\
&=\inf_{\bm{v}\in S^{j}}\sqrt{\sum_{i=1}^{j+1}\innerproduct{(A_{1}-A_{2})\bm{v}}{\bm{e}_{i}}^{2}}=\inf_{\bm{v}\in S^{j}}\sqrt{0+\sum_{i=2}^{j+1}\innerproduct{(A_{1}-A_{2})\bm{v}}{\bm{e}_{i}}^{2}}\\
&=\inf_{\bm{v}\in S^{j}}\sqrt{\sum_{i=2}^{j+1}\innerproduct{\bm{v}}{\bm{e}_{i}}^{2}(\innerproduct{A_{1}\bm{e}_{i}}{\bm{e}_{i}}-\innerproduct{A_{2}\bm{e}_{i}}{\bm{e}_{i}})^{2}}
\end{align*}
\indent while for $\tilde{A_{1}},\tilde{A}_{2}$ we have :
\begin{align*}
d(\tilde{A_{1}},\tilde{A_{2}})&=\inf_{\bm{v}\in S^{j-1}}\sqrt{\sum_{i=1}^{j}\innerproduct{(\tilde{A_{1}}-\tilde{A_{2}})\bm{v}}{\hat{f}(\bm{\varepsilon}_{i})}^{2}}\\
&=\inf_{\bm{v}\in S^{j-1}}\sqrt{\sum_{i=1}^{j}\innerproduct{\bm{v}}{\hat{f}(\bm{\varepsilon}_{i})}^{2}(\innerproduct{\tilde{A_{1}}\hat{f}(\bm{\varepsilon}_{i})}{\hat{f}(\bm{\varepsilon}_{i})}-\innerproduct{\tilde{A_{2}}\hat{f}(\bm{\varepsilon}_{i})}{\hat{f}(\bm{\varepsilon}_{i})})^{2}}\\
&=\inf_{\bm{v}\in S^{j-1}}\sqrt{\sum_{i=1}^{j}\innerproduct{\bm{v}}{\hat{f}(\bm{\varepsilon}_{i})}^{2}(\innerproduct{\hat{f}(A_{1}\bm{\varepsilon}_{i})}{\hat{f}(\bm{\varepsilon}_{i})}-\innerproduct{\hat{f}(A_{2}\bm{\varepsilon}_{i})}{\hat{f}(\bm{\varepsilon}_{i})})^{2}}
\end{align*}
\indent In other words, $d(A_{1},A_{2})=d(\tilde{A_{1}},\tilde{A_{2}})$. Therefore, for every $\varepsilon>0$, we have $\delta=\varepsilon>0$ satisfying that $d(A_{1},A_{2})<\delta \Rightarrow d(\phi_{\bm{x}}(A_{1}),\phi_{\bm{x}}(A_{2}))<\varepsilon,\forall{A_{1},A_{2}}\in\pi^{-1}(\bm{x})$ , i.e. $\phi_{\bm{x}}$ is continuous when restricted on $\pi^{-1}(\bm{x})$.\\
\indent Now for every $A,B\in\pi^{-1}(\mathcal{N}_{\bm{x}})$, by the fact that $\varphi$ is surjective, there exists $C\in\pi^{-1}(\mathcal{N}_{\bm{x}})$ satisfying the following two requirements:\\
\indent \indent 1. $\phi_{C\bm{e}_{1}}(C)=\phi_{B\bm{e}_{1}}(B)$, i.e. $\varphi(C)$ and $\varphi(B)$ only differs in the first component.\\
\indent \indent 2. $C\bm{e}_{1}=A\bm{e}_{1}$, i.e. $C\in\pi^{-1}(A\bm{e}_{1})$\\
\indent By the first requirement,
\begin{align*}
d(C,B)&=\inf_{\bm{v}\in S^{j+1}}\sqrt{\innerproduct{\bm{v}}{\bm{e}_{1}}^{2}(\innerproduct{C\bm{e}_{1}}{\bm{e}_{1}}-\innerproduct{B\bm{e}_{1}}{\bm{e}_{1}})^{2}}\\
&=\sqrt{(\innerproduct{C\bm{e}_{1}}{\bm{e}_{1}}-\innerproduct{B\bm{e}_{1}}{\bm{e}_{1}})^{2}}\\
&=\norm{(C-B)\bm{e}_{1}}=\norm{\pi(C)-\pi(B)}\\
&=\norm{\varphi(C)-\varphi(B)} \;\;\;\;\textit{(from\,the\,vanished\,difference\,of\,the\,second\,component)}\\
&=d(\varphi(C),\varphi(B))
\end{align*}
\indent By the second requirement, $d(\varphi(A),\varphi(C))=d(A,C)$ as we've computed.\\
\indent With the induced metric on the product space $\mathcal{N}_{\bm{x}} \times \mathrm{O}(j)$, it computes
\begin{align*}
d(\varphi(A),\varphi(B))&=\sqrt{d^{2}(\varphi(A),\varphi(C))+d^{2}(\varphi(C),\varphi(B))}\\
&=\sqrt{d^{2}(A,C)+d^{2}(C,B)} \;\;\;\;\textit{(substitution\,of\,equal\,quantity)}\\
&=d(A,B)
\end{align*}
\indent Thus $\varphi$ is an isometry, i.e. the homeomorphism we desire.

\end{proof}
However the infinite cases are not so simple. The following is a natural generalization of spheres :
\begin{defi}(\textbf{Cardinal-Dimensional Sphere : }$S^{\kappa}$)\\
\indent Let $\kappa \in \mathbf{Card}$ and $\mathcal{H}$ be a real Hilbert space with $\dim \mathcal{H} = |\kappa + 1| \in \mathbf{Card}$. Then $S^{\kappa}$ is defined as the unit sphere in $H$, i.e.
\[
S^{\kappa}=\{\bm{x} \in \mathcal{H} : \norm{\bm{x}}=1\}
\]
\end{defi}
Notice that $\mathbb{R}^{n+1}$ is also a finite-dimensional real Hilbert space, then for finite dimension the definition above gives out exactly $S^{n}$ which shows that the definition above generalizes the finite-dimensional cases. Similarly we define the $\kappa$-dimensional open disc $D^{\kappa}$ and the closed one $\overline{D^{\kappa}}$.\\
To have a further discussion, the following definition and lemmas are necessary.
\begin{defi}(\textbf{$\lambda$-Shift Operator} $\mathscr{S}_{\{\bm{e}_{\alpha}\},\lambda}$):
Let $\mathcal{H}$ be a $\kappa$-dimensional real Hilbert space with an orthonormal basis $\{\bm{e}_{\alpha}\}_{\alpha=0}^{\kappa}$. For $\lambda\in\mathbb{N}$ the \textbf{$\lambda$-shift operator $\mathscr{S}_{\{\bm{e}_{\alpha}\},\lambda}$} is defined as
\begin{align*}
\mathscr{S}_{\{\bm{e}_{\alpha}\},\lambda} : \mathcal{H} &\rightarrow \mathcal{H}\\
\bm{x} &\mapsto \sum_{\alpha=0}^{\kappa}\innerproduct{\bm{x}}{\bm{e}_{\alpha}}\bm{e}_{\alpha+\lambda}
\end{align*}
Particularly, for $\lambda=1$ and $\{\bm{e}_{\alpha}\}$ clearly selected, the $1$-shift operator $\mathscr{S}_{\{\bm{e}_{\alpha}\},1}$ is denoted as the \textbf{shift operator $\mathscr{S}$} which has the form:
\begin{align*}
\mathscr{S} : \mathcal{H} &\rightarrow \mathcal{H}\\
\bm{x} &\mapsto \sum_{i=0}^{\infty}\innerproduct{\bm{x}}{\bm{e}_{i}}\bm{e}_{i+1}+\sum_{\alpha=\omega}^{\kappa}\innerproduct{\bm{x}}{\bm{e}_{\alpha}}\bm{e}_{\alpha}
\end{align*}
\end{defi}
Next lemma gives out the reason why we can't simply generalize the above lemma about fibration by just repeating the steps:
\begin{lemma}
For cardinal number $\kappa \geq \omega$, $S^{\kappa}$ is contractible.
\end{lemma}
\begin{proof}
Treat $S^{\kappa}$ as a embed subset of the real Hilbert space $\mathcal{H}$ with an orthonormal basis $\{\bm{e_{\alpha}}\}_{\alpha\in\mathfrak{A}}$. For each point $\bm{x} \in S^{\kappa}$, $\bm{x}$ has the form $\sum_{\alpha\in\mathfrak{A}}\innerproduct{\bm{x}}{\bm{e}_{\alpha}}\bm{e}_{\alpha}$. It's direct to see $\innerproduct{\mathscr{S}\bm{x}}{\bm{e}_{0}}=0$, where $\mathscr{S}$ is the shift operator.\\
Here we construct a homotopy between the identity map $I$ of $S^{\kappa}$ and the shift operator $\mathscr{S}$ restricted on $S^{\kappa}$:\begin{align*}
F_{1}(t,\bm{x}) : [0,1] &\times S^{\kappa} \rightarrow S^{\kappa} \\
(t,\bm{x}) &\mapsto \frac{t \mathscr{S}\bm{x} + (1-t) \bm{x}}{\norm{t \mathscr{S}\bm{x} + (1-t)\bm{x}}}
\end{align*}
It's necessary to show that $\norm{t \mathscr{S}\bm{x} + (1-t) \bm{x}}\neq{0}$. First, for the cases $t=0$ the mapping degenerates to $I$, therefore in the following we assume that $t$ is non-zero. Besides, if $\innerproduct{\bm{x}}{\bm{e}_{0}}\neq{0}$ then $t \mathscr{S}\bm{x} + (1-t) \bm{x}$ can't be the zero vector $\bm{0}$ thus whose norm can't be $0$. Therefore we can assume that $\innerproduct{\bm{x}}{\bm{e}_{0}}=0$. For the same reason we assume that $\sum_{\alpha=\omega}^{\kappa}\innerproduct{\bm{x}}{\bm{e}_{\alpha}}=\bm{0}$. Then via calculation under the assumptions, $\innerproduct{t \mathscr{S}\bm{x} + (1-t) \bm{x}}{\bm{e}_{i}}=t\innerproduct{\bm{x}}{\bm{e}_{i+1}}+(1-t)\innerproduct{\bm{x}}{\bm{e}_{i}},i\in\mathbb{N}$. 
The value above be $0$ if and only if $\innerproduct{\bm{x}}{\bm{e}_{i+1}}=\frac{t-1}{t}\innerproduct{\bm{x}}{\bm{e}_{i}},i\in\mathbb{N}$. Thus under assumption $\norm{\mathscr{S}\bm{x}}=\frac{t-1}{t}\norm{\bm{x}}$ which contradicts to the fact that $\norm{\mathscr{S}\bm{x}}=\norm{\bm{x}}=1$. In brief, $\norm{t \mathscr{S}\bm{x} + (1-t) \bm{x}}\neq{0}$.\\
In addition, take $\bm{x}_{0}=\bm{e}_{0} \in S^{\kappa}$ as the base point, the following is a homotopy between the mentioned shift operator $\mathscr{S}$ and a constant mapping $c:S^{\kappa}\rightarrow\{\bm{e}_{0}\}$: \begin{align*}
F_{2}(t,\bm{x}) : [0,1] \times S^{\kappa} &\rightarrow S^{\kappa} \\
(t,\bm{x}) &\mapsto t \bm{e}_{0} + \sqrt{1-t^{2}}\mathscr{S}\bm{x}
\end{align*}
To show that the value of $F_{2}(t,\bm{x})$ is in $S^{\kappa}$, we only need to verify the norm. Notice that due to the shift operator $\mathscr{S}$ we have $\innerproduct{\mathscr{S}\bm{x}}{\bm{e}_{0}}=0$, therefore $\norm{t \bm{e}_{0} + \sqrt{1-t^{2}}\mathscr{S}\bm{x}}=\sqrt{t^{2}+(\sqrt{1-t^{2}})^{2}}=1$ as what we desire.\\
The two homotopy $F_{1}(\bm{x},t)$ and $F_{2}(\bm{x},t)$ together implies that the identity map of $S^{\kappa}$ namely $I$ is homotopic to a constant mapping $c:S^{\kappa}\rightarrow\{\bm{e}_{0}\}$, i.e. $S^{\kappa}$ is contractible.
\end{proof}
Therefore, since to have a homeomorphism it's necessary to have homotopy equivalence, such structures no longer survives. However, some of the following theorems might help in further discussions (than all what we'll give out in this paper).

\begin{lemma}
The $\kappa$-sphere $S^{\kappa} (\kappa\geq\aleph_{0})$ is a topological manifold based on Hilbert space with the subspace topology induced from the norm topology of the $\kappa$-dimensional real Hilbert space $\mathcal{H}$.
\end{lemma}
\begin{proof}
For convenience we take a part of $S^{\kappa}$ via an orthonormal basis $\{\bm{e}_{\alpha}\}_{\alpha=0}^{\kappa}$ as the set : \[(S^{\kappa})^{+}=\{\bm{x}\in S^{\kappa} : \innerproduct{\bm{x}}{\bm{e}_{0}}>0\}\]
Since via basis transformation the whole $S^{\kappa}$ can be covered, the choice of the basis doesn't matter a lot.
Take a subspace of $\mathcal{H}$ as follows:\[\mathscr{S}\mathcal{H}=\{\bm{x}\in \mathcal{H} : \innerproduct{\bm{x}}{\bm{e}_{0}}=0\}=(\mathsf{span}(\{\bm{e}_{0}\}))^{\perp}\]
Define the projection $P$ as follows:
\begin{align*}
P : (S^{\kappa})^{+} &\rightarrow D^{\kappa} \subset \mathscr{S}\mathcal{H} \\
\bm{x} &\mapsto \bm{x}-\innerproduct{\bm{x}}{\bm{e}_{0}}\bm{e}_{0}
\end{align*}
Obviously $P^{2}=P$ and $P$ is a bijection whose inverse is :
\begin{align*}
P^{-1} : & D^{\kappa} \rightarrow (S^{\kappa})^{+} \\
\bm{y} &\mapsto \bm{y}+\sqrt{1-\norm{\bm{y}}^{2}}\bm{e}_{0}
\end{align*}
Notice that by open-ball-covering, $(S^{\kappa})^{+}$ forms a open subset of $S^{\kappa}$ while the image of $P$ is the open disc in $\mathscr{S}\mathcal{H}$ namely $\{\bm{x} \in \mathscr{S}\mathcal{H} : \norm{\bm{x}}<1\}=D^{\kappa}$. In other words, the domain and range of $P$ are open.\\
Since $P$ is a projection, we have $\norm{P\bm{x}}\leq\norm{\bm{x}}$, for $H$ is a metric space thus $P$ is continuous.\\
To show that $P^{-1}$ is continuous, we only need verify $P^{-1}$ is continuous on every point in its domain $P(S^{\kappa})^{+}$. Equivalently speaking, for $\bm{y}_{0}\in P(S^{\kappa})^{+}$, for any positive real number $\varepsilon$ , $\norm{P^{-1}\bm{y}_{0}-P^{-1}\bm{y}}<\varepsilon$ implies that there exists real positive number $\delta$ satisfies $\norm{\bm{y}_{0}-\bm{y}}<\delta$ :
\[
\norm{\bm{y}_{0}-\bm{y}}=\norm{P\bm{x}_{0}-P\bm{x}}=\norm{P(\bm{x}_{0}-\bm{x})}\leq\norm{\bm{x}_{0}-\bm{x}}<\varepsilon
\]
in which expression let $\delta$ be $\delta=\varepsilon$ satisfies the requirement.\\
Therefore, $P$ is a continuous bijection whose inverse $P^{-1}$ is also continuous, i.e. $P$ is a homeomorphism between $(S^{\kappa})^{+}$ and an connected open subset of $H$. Recall that we've been using the basis $\{\bm{e}_{\alpha}\}_{\alpha=0}^{\kappa}$. For any point $\bm{x}'$ in $S^{\kappa}$, use such a basis containing $\bm{x}'$  thus $(S^{\kappa})^{+}$ be a neighbourhood of $\bm{x}'$ which shows that $S^{\kappa}$ is a topological manifold.
\end{proof}
\begin{corollary}
$S^{\kappa}\,(\kappa\geq 1)$ is a product topological manifold.
\end{corollary}
For a stronger version with differentiability, we need the following lemma to give out the completeness as preparation:
\begin{lemma}
$S^{\kappa}$ is complete with the metric induced by the norm $\norm{\bullet}_{\mathcal{H}}$.
\end{lemma}
\begin{proof}
Treat $S^{\kappa}$ as a subset of $\mathcal{H}$ thus $S^{\kappa}$ is the boundary of the open disc $D^{\kappa}=\{\bm{x} \in \mathcal{H} : \norm{\bm{x}}<1\}$, therefore $S^{\kappa}$ is a closed subset of $\mathcal{H}$. Since $\mathcal{H}$ is a Hilbert space, $\mathcal{H}$ is a complete metric space. In brief, $S^{\kappa}$ is complete with the metric inherited from $\mathcal{H}$, i.e. the metric induced from the norm $\norm{\bullet}_{\mathcal{H}}$.
\end{proof}
As a corollary, this lemma follows:
\begin{corollary}
The $\kappa$-sphere $S^{\kappa} (\kappa\geq\aleph_{0})$ is a Fr\'{e}chet-differentiable manifold based on Hilbert space.
\end{corollary}
\begin{proof}
Take any two intersected atlases of $S^{\kappa}$, namely $(U,P_{U})$ and $(V,P_{V})$ with $U \cap V \neq \emptyset$. To research the transition map $P_{U} \circ P_{V}^{-1}$, we'd like to give out $P_{U}$ and $P_{V}^{-1}$ (but restricted on $U \cap V$) and their Fr\'{e}chet derivative.
\begin{align*}
P_{U} : {U}\cap{V} &\rightarrow D^{\kappa} \\
\bm{x} &\mapsto \bm{x}-\innerproduct{\bm{x}}{\bm{e}_{1}}\bm{e}_{1}
\end{align*}
\begin{align*}
P_{V}^{-1} : & D^{\kappa} \rightarrow {U}\cap{V} \\
\bm{y} &\mapsto \bm{y}+\sqrt{1-\norm{\bm{y}}^{2}}\bm{e}_{2}
\end{align*}
Now calculate their Fr\'{e}chet derivative (which implying that they're Fr\'{e}chet-derivable).\\
Since $P_{U}$ is linear, its Fr\'{e}chet derivative at any point in its domain is $P_{U}$ but whose domain is $\mathcal{H}$.\\
Now for $P_{V}^{-1}$. Notice the identity $\norm{\bm{x}}=\sqrt{\innerproduct{\bm{x}}{\bm{x}}}\forall\bm{x}\in\mathcal{H}$.
\begin{align*}
\lim_{\bm{h}\rightarrow\bm{0}}\frac{\norm{\innerproduct{\bm{x}+\bm{h}}{\bm{x}+\bm{h}}-\innerproduct{\bm{x}}{\bm{x}}-2\innerproduct{\bm{x}}{\bm{h}}}_{\mathbb{R}}}{\norm{\bm{h}}_{\mathcal{H}}}=\lim_{\bm{h}\rightarrow\bm{0}}\frac{\norm{\innerproduct{\bm{h}}{\bm{h}}}_{\mathbb{R}}}{\norm{\bm{h}}_{\mathcal{H}}}=0
\end{align*}
Therefore by definition we have 
\begin{align*}
D\norm{\bullet}^{2}(\bm{x}):\bm{h}\mapsto2\innerproduct{\bm{x}}{\bm{h}}
\end{align*}
Again with the chain rule of Fr\'{e}chet derivatives, it follows :
\begin{align*}
DP_{V}^{-1}(\bm{y}):\bm{h}\,\mapsto&\,\bm{h}+\frac{d\sqrt{1-\norm{\bm{y}}^{2}}}{d\norm{\bm{y}}^{2}} \cdot (D\norm{\bullet}^{2}(\bm{y}))(\bm{h}) \cdot \bm{e}_{2}\\
&\;=\bm{h}+(-\frac{1}{2\sqrt{\bm{1-\norm{\bm{y}}^{2}}}}) \cdot 2\innerproduct{\bm{y}}{\bm{h}} \cdot \bm{e}_{2}\\
&\;=\bm{h}-\frac{\innerproduct{\bm{y}}{\bm{h}}}{\sqrt{\bm{1-\norm{\bm{y}}^{2}}}}\bm{e}_{2}
\end{align*}
where the denominator in the expression won't vanish when restricted on ${U}\cap{V}$. \\
Still by chain rule of Fr\'{e}chet derivative, now we can give out the Fr\'{e}chet derivative of $P_{U} \circ P_{V}^{-1}$ as $D(P_{U} \circ P_{V}^{-1})(\bm{x})=DP_{U}(P_{V}^{-1}(\bm{x})) \circ DP_{V}^{-1}(\bm{x})$.
\begin{align*}
D(P_{U} \circ P_{V}^{-1})(\bm{y}) : \bm{h} \mapsto P_{U}(\bm{h}-\frac{\innerproduct{\bm{y}}{\bm{h}}}{\sqrt{\bm{1-\norm{\bm{y}}^{2}}}}\bm{e}_{2})=P_{U}(\bm{h})-\frac{\innerproduct{\bm{y}}{\bm{h}}}{\sqrt{\bm{1-\norm{\bm{y}}^{2}}}}P_{U}(\bm{e}_{2})
\end{align*}
As a conclusion we desire, $P_{U} \circ P_{V}^{-1}$ is Fr\'{e}chet-derivable. Due to the arbitrariness of the choice of the atlases, transition maps of $S^{\kappa}$ are all Fr\'{e}chet-derivable, i.e. $S^{\kappa}$ is a differential Hilbert manifold.We hope the discussions above would help to add a manifold structure to $O(\kappa)$.
\end{proof}
Now get back to $O(\kappa)$ turn our sight to the group actions.
Denote the subgroup which keeps $\mathsf{span}(\{\bm{e}_{i}\}_{i=0}^{n-1})$ fixed as $O^{(n)}(\kappa)$ (thus $O^{(0)}(\kappa)=O(\kappa)$) and the subgroup which keeps $(\mathsf{span}(\{\bm{e}_{i}\}_{i=0}^{n-1}))^{\perp}$ fixed as $O_{(n)}(\kappa)$.
\begin{theorem}
Given a cardinal number $\kappa \geq \omega$ and a real Hilbert space $\mathcal{H}$ s.t. $\dim\mathcal{H}=\kappa$ with an orthonormal basis $\{\bm{e}_{i}\}_{i\in\kappa'}$. Let $O(\kappa)$ acts on $S^{\kappa}$ by left multiplication. Let $O^{(1)}(\kappa)$ acts on $O(\kappa)$ on the right side. We have :
\[
O(\kappa)/O^{(1)}(\kappa)\approx{S^{\kappa}}
\]
\end{theorem}
\begin{proof}
Let $O(\kappa)$ acts on $S^{\kappa}$ on the left transitively :
\begin{align*}
\lambda(-,-) : O(\kappa) \times S^{\kappa} &\to S^{\kappa}\\
(A,\bm{x}) &\mapsto A\bm{x}
\end{align*}
Thus the curried one $\lambda(-,\bm{e}_{0})$ gives out a surjection. Notice that both spaces are equipped with the topology induced by metrics (from norms). Given any operator $A \in O(\kappa)$, check its open $\varepsilon$-neighbourhood ($\varepsilon>0$) that given any $B \in O(\kappa)$ s.t. $\norm{A-B}<\varepsilon$ :
\[
\norm{(A-B)\bm{e}_{0}}\leq\norm{A-B}\norm{\bm{e}_{0}}=\norm{A-B}
\]
That is,
\[
d\left(\lambda\left(A,\bm{e}_{0}\right),\lambda\left(B,\bm{e}_{0}\right)\right) \leq d(A,B)<\varepsilon
\]
In other words, $\lambda(-,\bm{e}_{0})$ is continuous.\\
Write $A$ into the form (and similar for $B$) :
\begin{align*}
A : \mathcal{H} &\to \mathcal{H}\\
\bm{x} &\mapsto \sum_{i\in\kappa'}\innerproduct{\bm{x}}{\bm{e}_{i}}(A\bm{e}_{i})
\end{align*}
Therefore, when and only when $d(A,B)=d\left(\lambda\left(A,\bm{e}_{0}\right),\lambda\left(B,\bm{e}_{0}\right)\right)$ (i.e. in the extreme cases), $A\bm{e}_{i}=B\bm{e}_{i}\forall i\geq{1}$. To be precise, $\lambda(-,\bm{e}_{0})$ maps the open $\varepsilon$-neighbourhood of $A$ onto the open $\varepsilon$-neighbourhood of $A\bm{e}_{0}$. As a result, $\lambda(-,\bm{e}_{0})$ is a continuous surjective open mapping, i.e. $\lambda(-,\bm{e}_{0}) : O(\kappa) \to S^{\kappa}$ is a quotient mapping. The resulting quotient space $O(\kappa)/\lambda(-,\bm{e}_{0})$ is thus homeomorphic to $S^{\kappa}$.\\

Now turn to consider the orbit space $O(\kappa)/O^{(1)}(\kappa)$. The right action is given as :
\begin{align*}
\rho(-,-) : O(\kappa) \times O^{(1)}(\kappa) &\to O(\kappa)\\
(X,A) &\mapsto XA^{-1}
\end{align*}
Two operators $A,B \in O(\kappa)$ lie in the same orbit if and only if $A^{-1}B \in O^{(1)}(\kappa)$, i.e. $\text{Orb}_{\rho}(A)=AO^{(1)}(\kappa)=BO^{(1)}(\kappa)=\text{Orb}_{\rho}(B)$ if and only if $A^{-1}B \in O^{(1)}(\kappa)=x\text{Stab}_{\lambda}(\bm{e}_{0})$.\\
Therefore, followed by a simple calculation :
\[
A\bm{e}_{0}=A(A^{-1}B\bm{e}_{0})=AA^{-1}B\bm{e}_{0}=B\bm{e}_{0}
\]
which implies $O(\kappa)/\lambda(-,\bm{e}_{0})=O(\kappa)/O^{(1)}(\kappa)$. As we've proved that the former is homeomorphic to $S^{\kappa}$, the later follows that $O(\kappa)/O^{(1)}(\kappa) \approx S^{\kappa}$
\end{proof}
By a little modification up to isomorphism (of topological groups) it follows :
\begin{corollary}
Given a real Hilbert space $\mathcal{H}$ s.t. $\dim\mathcal{H}=\kappa \geq \omega$,
\[
O^{(n)}(\kappa)/O^{(n+1)}(\kappa)\approx{S^{\kappa}}
\]
\end{corollary}
What's more,
\begin{corollary}
For all $\kappa\in\mathbf{Card}$, the following is a fibration :
\[
\xymatrix{
O^{(1)}(\kappa) \ar[rr] && O(\kappa) \ar[rr]^{\tilde{\pi}} && S^{\kappa}
}
\]
Here $\pi$ is the canonical projection induced by the quotient, $\sigma$ is the homeomorphism $O(\kappa)/O^{(1)}(\kappa) \to S^{\kappa}$ and the fibre projection is defined as $\tilde{\pi}=\sigma\circ\pi$.
\end{corollary}
\begin{proof}
Our goal is : for every $X \in O(\kappa)/O^{(1)}(\kappa)$, take a appropriate open neighbourhood $\mathcal{N}$ of $X$, find a homeomorphism $\varphi$ to make the diagram commute :
\[
\xymatrix{
O(\kappa)\arsupset&\pi^{-1}(\mathcal{N}) \ar@{-->}[rr]^{\varphi}\ar[dr]^{\pi}&&\mathcal{N} \times O^{(1)}(\kappa) \ar[dl]_{\text{proj}_{1}}\\
&&\mathcal{N} \arsubset & S^{\kappa} \arapprox&O(\kappa)/O^{(1)}(\kappa)
}
\]
Here we start from taking a distinguished element for $X$, that is, such an injection :
\begin{align*}
h_{0}(-) : O(\kappa)/O^{(1)}(\kappa) &\to O(\kappa)\\
X &\mapsto h_{0}(X) \;\mathrm{s.t.}\; \pi(h_{0}(X))=X
\end{align*}
By the fact that the orbits cuts $O(\kappa)$ and with the axiom of choice, such an injection exists. Recall that for every $g_{1},g_{2} \in O(\kappa)$, $\pi(g_{1})=\pi(g_{2})$ if and only if $g_{1}^{-1}g_{2} \in O^{(1)}(\kappa)$, hence for every $g \in O(\kappa)$, $g \in X$ if and only if $g^{-1}\,h_{0}(X) \in O^{(1)}(\kappa)$. By taking $\phi : O(\kappa) \to O^{(1)}(\kappa);g \mapsto g^{-1}\,h_{0}(X)$, the following diagram is commutative :
\[
\xymatrix{
g\in\pi^{-1}(X) \ar@{|->}[rr]^{(\pi,\phi)\circ\Delta}\ar@{|->}[dr]_{\pi}&&(X,g^{-1} \, h_{0}(X)) \ar@{|->}[dl]^{\text{proj}_{1}}\\
&X
}
\]
where
\begin{align*}
((\pi,\phi)\circ\Delta)^{-1}: O(\kappa)/O^{(1)}(\kappa) \times O^{(1)}(\kappa) &\to \mathcal{N}\\
(X,h) &\mapsto h_{0}(X)\,h^{-1}
\end{align*}
Let $\varphi=(\pi,\phi)\circ\Delta$ which is bijective. Now to show $\varphi$ is homeomorphism, we just need to explicitly construct a continuous $h_{0}(X)$. Consider the homeomorphism $\sigma^{-1}:S^{\kappa} \to O(\kappa)/O^{(1)}(\kappa)$. Then for every $\bm{x} \in S^{\kappa}$ we have $\sigma^{-1}(\bm{x}) \in O(\kappa)/O^{(1)}(\kappa)$. By taking $\bm{e}_{0} \in S^{\kappa}$, denote the reflection generated by $\bm{x},\bm{y} \in S(\kappa)$ as $r(\bm{x},\bm{y})$, then $-r(-\bm{e}_{0},\bullet):S^{\kappa} \to O(\kappa);\bm{x} \mapsto -r(-\bm{e}_{0},\bm{x})$ together with the homeomorphism $\sigma$ finishes the construction : $h_{0}(X)=-r(-\bm{e}_{0},\sigma(X))$, i.e. $h_{0}(\bullet)=-r(-\bm{e}_{0},\bullet)\circ\sigma$.
\end{proof}

\newpage
\section{Conjectures}
\indent\indent Though in this paper we give out $\Theta(\kappa)$ and $O(\kappa)$, namely `two' generalizations of the orthogonal groups, we are not sure whether $\Theta(\kappa)$ and $O(\kappa)$ are equal or not (for which reason we use the symbol `$\trianglelefteq$' rather than `$\triangleleft$' to express their relationship), requiring an explicit `exceptional' example. Therefore, the quotient group $O(\kappa) / \Theta(\kappa)$ seem to be an interesting topic in terms of not only algebra but also topology and even analysis. In any case, the topology of $O(\kappa)$ and $\Theta(\kappa)$ and their relationship to $\mathrm{O}$ needs further discussion. Since many of our tries failed, a larger number of more abstract viewpoints, more further knowledge and more powerful utilities are required.

\nocite{*}
\bibliographystyle{siam}
\bibliography{LB5}
\end{document}